\documentclass[11pt,a4paper,reqno]{amsart}
\usepackage{color}
\usepackage{amsmath}
\usepackage{amsfonts}
\usepackage{amssymb}
\usepackage{amsthm}
\usepackage{newlfont}
\usepackage{graphicx}

\newtheorem{thm}{Theorem}[section]
\newtheorem{cor}[thm]{Corollary}
\newtheorem{lem}[thm]{Lemma}
\newtheorem{prop}[thm]{Proposition}
\theoremstyle{definition}
\newtheorem{defn}[thm]{Definition}
\theoremstyle{remark}
\newtheorem{rem}[thm]{Remark}

\numberwithin{equation}{section}
\numberwithin{thm}{section}
\newcommand{\eps}{\varepsilon}

\newcommand{\lsm}{\lesssim}

\newcommand{\R}{{\mathbb{R}}}
\newcommand{\ed}{\end {document}}

\newcounter{smalllist}

\title[Dynamics for energy critical wave equation]{Dynamics for the energy critical nonlinear
wave equation in high dimensions}
\author{Dong Li}
\address{Institute for Advanced Study, 1st Einstein Drive, Princeton NJ, 08544}

\author{Xiaoyi Zhang}
\address{Institute for Advanced Study, 1st Einstein Drive, Princeton NJ, 08544.
Academy of Mathematics and System Sciences, Beijing 100080}

\begin{document}
\maketitle
\begin{abstract}
In \cite{duck-merle:wave}, T. Duyckaerts and F. Merle studied the
variational structure near the ground state solution $W$ of the
energy critical wave equation and classified the solutions with the
threshold energy $E(W,0)$ in dimensions $d=3,4,5$. In this paper, we
extend the results to all dimensions $d\ge 6$. The main issue in
high dimensions is the non-Lipschitz continuity of the nonlinearity
which we get around by making full use of the decay property of $W$.
\end{abstract}

\section{introduction}
We consider the Cauchy problem of the focusing energy critical
nonlinear wave equation:
\begin{equation}\label{nlw}
\begin{cases}
u_{tt}-\Delta u-|u|^{\frac 4{d-2}}u=0,  \\
u(0,x)=u_0(x),\ \partial_t u(0,x)=u_1(x).
\end{cases}
\end{equation}
where $u(t,x)$ is a real function on $\R\times\R^d$, $d\ge 3$ and
$u_0\in \dot H_x^1(\R^d)$, $u_1\in L_x^2(\R^d)$. The name "energy
critical" refers to the fact that the scaling
\begin{equation}\label{scaling}
u(t,x)\to
u_{\lambda}(t,x)=\lambda^{-\frac{d-2}2}u(\lambda^{-1}t,\lambda^{-1}x)
\end{equation}
leaves both the equation and the energy invariant. Here, the energy
is defined by
\begin{equation}\label{energy}
E(u(t),\partial_t u(t))=\frac 12\|\partial_t u(t)\|_2^2+\frac
12\|\nabla u(t)\|_2^2
-\frac{d-2}{2d}\|u(t)\|_{\frac{2d}{d-2}}^{\frac{2d}{d-2}},
\end{equation}
and is conserved in time.

From the classical local theory (cf. \cite{gsv:wave,
kapitanski, ls95, %pecher,
sha-stru, sha:book, sogge}), for any
$(u_0,u_1)\in \dot H^1_x\times L_x^2$, there exists a unique
maximal-lifespan solution of \eqref{nlw} on a time interval
$(-T_-,T_+)$ such that the local \emph{scattering size}
\begin{align*}
S_I(u):=\|u\|_{L_{t,x}^{\frac{2(d+1)}{d-2}}(I\times\R^d)}<\infty, \
\end{align*}
for any compact interval $I\subset(-T_-,T_+)$. If
$S_{[0,T_+)}(u)=\infty$, we say $u$ blows up forward in time.
Likewise $u$ blows up backward in time if $S_{(-T_+,0]}(u)=\infty$.
We also recall the fact that the non-blowup of $u$ in one direction
implies scattering in the space $\dot H_x^1\times L_x^2$ in that
direction.

For the defocusing energy critical NLW, the global wellposedness and
scattering for all finite energy solutions was established in
\cite{grillakis:3d, grillakis:hd, sha-stru, sha-stru:anna, sha:book,
merlekenig:wave}. In the focusing case, depending on the size of the
kinetic energy of the initial data, both scattering and blowup may
occur. The threshold between blowup and scattering is believed to be
determined by the ground state solution of the equation \eqref{nlw}:

$$
W(x)=\left(1+\frac{|x|^2}{d(d-2)}\right)^{-\frac{d-2}2},
$$
which solves the static nonlinear wave equation
\begin{align*}
 \Delta
W+W^{\frac{d+2}{d-2}}=0.
\end{align*}
This was verified by Kenig-Merle \cite{merlekenig:wave} in
dimensions $d=3,4,5$.

\begin{thm}[Global wellposedness and scattering \cite{merlekenig:wave}]
\label{gwp}  Let  $d= 3,4,5$ and $(u_0, u_1)\in \dot H_x^1\times
L_x^2$. Assume that $E(u_0,u_1)<E(W,0)$. Let $u=u(t,x)$ be the
maximal-lifespan solution of \eqref{nlw} on $I\times \R^d$.

i) If $\|\nabla u_0\|_2<\|\nabla W\|_2$, then $I=\R$ and the
scattering size of $u$ is finite,
\begin{align*}
S_I(u)=\|u\|_{L_{t,x}^{\frac{2(d+1)}{d-2}}(I\times\R^d)}<\infty.
\end{align*}

ii) If $\|\nabla u_0\|_2>\|\nabla W\|_2$, then $u$ blows up at
finite time in both time directions, i.e, $|I|<\infty$.

\end{thm}

\begin{rem} In \cite{merlekenig:wave}, Theorem \ref{gwp} was proved in
dimensions $d=3,4,5$. To generalize it to higher dimensions, one
needs a stronger stability result which has been worked out for
the energy critical NLS \cite{vz:unpublished}.
%or a weaker stability result \cite{tv:stability} (see \cite{kv:energy}
%for the treatment of the energy critical nonlinear Schr\"odinger
%equation).
After some changes, the stronger stability theory for NLS still
holds for NLW, and we will address this problem elsewhere
\cite{lz:prep}. We stress that Theorem 1.4 below depends crucially
on this generalization.
\end{rem}

Theorem \ref{gwp} confirmed that the threshold between blowup and
scattering is given by the ground state $W$. Our purpose of this
paper is not to investigate the global wellposedness and scattering
theory below the threshold. Instead, we aim to continue the study in
\cite{duck-merle:wave} on what will happen if the solution has the
threshold energy $E(W,0)$. In that paper, T. Duyckaerts and F. Merle
carried out a very detailed study of the dynamical structure around
the ground state solution $W$. They were able to give the
characterization of solutions with the threshold energy in
dimensions $d=3,4,5$.

In this paper, we aim to extend the results in
\cite{duck-merle:wave} to all dimensions $d\ge 6$. Although the
whole framework designed for low dimensions can also be used for the
high dimensional setting, there are a couple of places where the
arguments break down in high dimensions. Roughly speaking, this was mainly caused by
the fact that the nonlinearity is no longer Lipschitz continuous or superlinear in
Strichartz space $\dot S^1$ (see Section 2 for the definition).
In  the whole proof, there are mainly two places where the Lipschitz continuity
and superlinearity is heavily needed. The first is in the construction
of threshold solutions $W^{\pm}$ using a perturbed equation where Lipschitz continuity
is used for a contraction argument. The second is in showing rigidity properties of
$W^{\pm}$ where we need to show that the nonlinearity is superlinear with respect
to perturbations. The superlinearity is a crucial property needed for a bootstrap
argument which were used to show that solutions exponentially close to $W^{\pm}$ must
coincide with $W^{\pm}$ up to symmetries. To get around this problem, we will employ
a similar technique which we used for the corresponding problem of energy critical NLS \cite{lz:schrodinger}.
When constructing the threshold solution, instead of using the standard
Strichartz space, we will use the \emph{weighted Sobolev space}
$H^{m,m}$ (see next Section for the definition). The weighted space
$H^{m,m}$ turns out to be a natural space in which the nonlinearity of the
perturbed equation can be proved to be Lipschitz continuous. To show the
rigidity of the threshold solutions, we will show the
perturbed nonlinearity of the form $R(v+w^a) -R(w^a)$ \footnote{Here
$w^a$ is the difference between the threshold solution and the ground state $W$, see Lemma \ref{gain-decay}.}
has better decay than the perturbation $v$ which already has certain exponential decay.
By proving that the difference $w^a$ is in $H^{m,m}$, we are able to transform the
perturbed nonlinearity into a form which can be treated by using the dyadic
decomposition trick from \cite{duck-merle:wave}. The rigidity of the threshold
solutions then follows after several boostrap steps.

% The first advantage
% for doing so is that the nonlinearity of the perturbed equation can
% be proved to be Lipschitz continuous in $H^{m,m}$. The second
% advantage is that we can show the difference between the threshold
% solution $W^a$ and $W$ is actually very regular, and decays nicely
% in space. This property is crucial when we try to prove that after
% extracting the linear term, the perturbed nonlinearity decays
% superlinealy in time (see Lemma \ref{gain-decay}).  The
% superlinearity is needed to show the rigidity of the threshold
% solutions $W^{\pm}$.

\vspace{0.2cm}

In all, the material in this paper allows us to extend the argument
in \cite{duck-merle:wave} to all dimensions $d\ge 6$. This is the
following

\begin{thm}\label{exist-w} Let $d\ge 6$.
There exists a spherically symmetric solution $W^{\pm}$ of
\eqref{nlw} defined on the maximal-lifespan
$(-T_-(W^{\pm}),T_+(W^{\pm}))$ with initial data
$(W_0^{\pm},W_1^{\pm})\in \dot H_x^1\times L_x^2$ such that
\begin{align}
E(W,0)=E(W_0^{+},W_1^+)=E(W_0^-,W_1^-),\label{equal energy}\\
T_+(W^-)=T_+(W^+)=+\infty,\ \ \mbox{and } \lim_{t\to +\infty}
\|W^{\pm}(t)-W\|_{\dot H_x^1}=0,\label{converge}\\
\|\nabla W^-\|_2<\|\nabla W\|_2,\ T_-(W^-)=+\infty,\
S_{(-\infty,0) }(W^-)<\infty,\label{scatter back}\\
\|\nabla W^+\|_2>\|\nabla W\|_2, \ T_-(W^+)<\infty.\label{blowup
forw}
\end{align}
\end{thm}

Now we classify the solutions with the threshold energy. Since the
equation is invariant under several symmetries, we can determine the
solution only modulo these symmetries. Let $u(t,x)$, $v(t,x)$ be two
spacetime functions. When we say $u=v$ up to symmetries of the
equation, we mean there exist $t_0\in \R$, $x_0\in \R^d$,
$\lambda_0>0$, $c_0, c_1\in \{+1,-1\}$ such that
$$
u(t,x)=\frac{c_0}{\lambda_0^{\frac{d-2}2}} v(\frac{t_0+ c_1
t}{\lambda_0},\frac{x+x_0}{\lambda_0}).
$$
With this convention we have

\begin{thm}\label{u is w}(Dynamical classification at the critical
level). Let $d\ge 6$. Let $(u_0,u_1)\in \dot H^1_x\times L_x^2$ such
that
\begin{align}\label{w-energy}
E(u_0,u_1)=E(W,0).
\end{align}
Let $u$ be the solution of \eqref{nlw} with initial condition
$(u_0,u_1)$ and $I$ be its maximal-lifespan. Then we have the following

(a) If $\|\nabla u_0\|_2<\|\nabla W\|_2$, then $I=\R$. Moreover,
either $u=W^-$ up to symmetries of the equation, or
$S_{\R}(u)<\infty.$

(b) If $\|\nabla u_0\|_2=\|\nabla W\|_2$, then $u=W$ up to
symmetries of the equation.

(c) If $\|\nabla u_0\|_2>\|\nabla W\|_2$ and $u_0 \in L^2$, then $u=W^+$ up to
symmetries of the equation or $I$ is finite.
\end{thm}

The proof of Theorem \ref{exist-w} and \ref{u is w} will follow
roughly the same strategy as in \cite{duck-merle:wave}. Here we make
a remark about the proof of Theorem \ref{u is w}. The second point
(b) is a direct application of the variational characterization of $W$.
It only remains to prove (a) and (c). In \cite{duck-merle:wave}, a large portion of the
work was devoted to showing the exponential convergence of the
solution to $W$, which after several minor changes, also works for
higher dimensions. For this reason, we do not repeat that part of
the argument and build our starting point on the following

\begin{prop}[Exponential convergence to $W$ \cite{duck-merle:wave}]
\label{prop:exp} Let $u$ be the solution to \eqref{nlw} with initial
condition $(u_0, u_1)\in \dot H_x^1 \times L_x^2$ satisfying
$E(u_0,u_1)=E(W,0)$.

(a) In the case $\|\nabla u_0\|_2<\|\nabla W\|_2$, then $u$ exists
globally. Suppose also that $S_{(0,\infty)}(u)=\infty$, then
$S_{(-\infty,0)}(u)<\infty$ and there exist $\lambda_0>0$, $x_0\in
\R^d$, $c>0$ and $C>0$ such that
\begin{align}\label{sub-close}
\left\|\nabla
( u(t)-\lambda^{-\frac{d-2}2}W(\lambda_0^{-1} ({x+x_0})) ) \right\|_2+\|\partial_t
u (t)   \|_2\le Ce^{-ct}.
\end{align}

(b) In the case $\|\nabla u_0\|_2>\|\nabla W\|_2$, we also assume
that $u_0\in L^2_x$ and $u$ exists globally forward in time, then
there exist $c, C>0$ and $\lambda_0$, $x_0$ such that
\begin{align}
\left\|\nabla(
u(t)-\lambda_0^{-\frac{d-2}2}W(\lambda_0^{-1}({x+x_0})) ) \right\|_2+\|\partial_t
u(t)  \|_2\le Ce^{-ct}.
\end{align}
\end{prop}

This paper is organized as follows. In Section 2, we introduce some
notations and collect some basic estimates. Section 3 is devoted to
proving Theorem \ref{exist-w}. In Section 4, we prove two useful
estimates. In Section 5, we use the two estimates to finish the
proof of Theorem \ref{u is w} by assuming Proposition
\ref{prop:exp}.

\subsection*{Acknowledgements} Both authors were supported by the National
Science Foundation under agreement No. DMS-0635607.  X.~Zhang was
also supported by NSF grant No.~10601060 and project 973 in China.

%As pointed in \cite{duck-merle}, the second point come directly from
% the following variational characterization of
%$W$, thus does not need the radial assumption.

%\begin{thm}\label{w-like}\cite{tttt}
%Let $c(d)$ denote the sharp constant in Sobolev-embedding
%$$
%\|f\|_{\frac{2d}{d-2}}\le c(d)\|\nabla f\|_2.
%$$
%Then the equality holds iff $f$ is $W$ up to symmetries. More
%precisely, there exists $(\theta_0,\lambda_0,x_0)\in \mathbb R\times
%\R^+\times \R^d$ such that
%$$
%f(x)=e^{i\theta_0}\lambda_0^{-\frac{d-2}2}W(\frac
%{x-x_0}{\lambda_0}).
%$$
%In particular, if
% $u_0$ satisfies
%$$
%E(u_0)=E(W),\ \|\nabla u_0\|_2=\|\nabla W\|_2,
%$$
%then $u_0$ coincides with $W$ up to symmetries, hence the
%corresponding solution $u$ coincides with $W$ up to symmetries.
%\end{thm}

%In the analysis of Duck-Merle, the only place that rely heavily on
%the low dimension is analysis is ??. Roughly speaking, in low
%dimensions, the nonlinearity is Lip continuous in $S^1$ space, while
%in the high dimensions.

%The organization of the paper is the following: In Section 2, we
%introduce preliminary results; In Section 3, we roughly explain the
%idea of the proof. In Section 4, we prove the existence of $W^-$ and
%$W^+$. In Section 5, we give the classification of the solution with
%critical energy.

\section{Preliminaries}

We use $X \lesssim Y$ or $Y \gtrsim X$ whenever $X \leq CY$ for some
constant $C>0$. We use $O(Y)$ to denote any quantity $X$ such that
$|X| \lesssim Y$. We use the notation $X \sim Y$ whenever $X
\lesssim Y \lesssim X$.  We will add subscripts to $C$ to indicate
the dependence of $C$ on the parameters. For example, $C_{i,j}$
means that the constant $C$ depends on $i,j$. The dependence of $C$
upon dimension will be suppressed.

We use the `Japanese bracket' convention $\langle x \rangle := (1
+|x|^2)^{1/2}$.

Throughout this paper, we will use $p_c$ to denote the total power
of nonlinearity:
$$
p_c=\frac{d+2}{d-2}.
$$

We write $L^q_t L^r_{x}$ to denote the Banach space with norm
$$ \| u \|_{L^q_t L^r_x(\R \times \R^d)} := \Bigl(\int_\R \Bigl(\int_{\R^d}
|u(t,x)|^r\ dx\Bigr)^{q/r}\ dt\Bigr)^{1/q},$$ with the usual
modifications when $q$ or $r$ are equal to infinity, or when the
domain $\R \times \R^d$ is replaced by a smaller region of spacetime
such as $I \times \R^d$.  When $q=r$ we abbreviate $L^q_t L^q_x$ as
$L^q_{t,x}$.

For any $s\in\R$, we define fractional derivative $|\nabla|^s$ via
the Fourier transform
\begin{align*}
\widehat{|\nabla|^s f}(\xi)=(4\pi^2|\xi|^2)^{\frac s2}\hat f(\xi).
\end{align*}
and define $\dot W_x^{s,p}(\R^d)$ to be Sobolev space with the norm
\begin{align*}
\|f\|_{\dot W_x^{s,p}}=\||\nabla|^s f\|_p.
\end{align*}
When $p=2$, we write $\dot W_x^{s,2}$ as $\dot H_x^s$.

\subsection{Fractional chain rule}
We record the following results from \cite{kpv93}.
\begin{lem}[Fractional chain rule]\label{chr}
Let $p_1,p_2,p_3,p_4\in (1,\infty)$, $p_1\in (1,\infty]$, be such
that $\frac 1{p_1}+\frac 1{p_2}=\frac 1p$, $\frac 1{p_3}+\frac
1{p_4}=\frac 1p$. Let $F\in C^1(\R)$ be such that $F(0)=0$, then we
have
\begin{align}
\||\nabla|^{\frac 12}(fg)\|_{p}\lsm \|f\|_{p_1}\||\nabla|^{\frac
12}g\|_{p_2}+\||\nabla|^{\frac 12} f\|_{p_3}\|g\|_{p_4}.\\
\||\nabla|^{\frac 12}F(f)\|_{p}\lsm\|F'(f)\|_{p_1}\||\nabla|^{\frac
12} f\|_{p_2}.
\end{align}
\end{lem}

\subsection{Linear wave equation and Strichartz estimates}

\begin{defn}[Wave-admissible pair]\label{admissible}
Let $d\ge 6$, we say the couple $(q,r)$ is admissible if $2\le
q\le\infty$ and
\begin{align*}
\frac 2q=(d-1)(\frac 12-\frac 1r).
\end{align*}
\end{defn}
Denote $I$ be a time slab and $\beta(r)=\frac{d+1}2(\frac 12-\frac
1r)$. We define $\dot S^1(I)$, $\dot N^1(I)$ to be the Banach space
with the norm
\begin{align*}
\|u\|_{\dot S^1(I)} &=\sup_{(q,r)\text{ admissible}} (\|u\|_{L_t^q\dot
W_x^{1-\beta(r),r}(I\times\R^d)}+\|\partial_t u\|_{L_t^q\dot W_x^{-\beta(r),r}(I\times\R^d)}),\\
\|u\|_{\dot N^1(I)} &=\inf_{u=u_1+u_2}
\|u_1\|_{L_t^{1}L_x^2(I\times\R^d)}+\||\nabla|^{\frac
12}u_2\|_{L_{t,x}^{\frac{2(d+1)}{d+3}}(I\times\R^d)}.
\end{align*}
With these notations, we write Strichartz inequalities for linear
wave equation as follows

\begin{lem}[Strichartz estimate \cite{gv95,ls95,tao:keel}] Let $f\in \dot H^1_x$, $g\in L_x^2$. Let $I$ be a time slab containing
$t_0$. Let $F\in \dot N^1(I)$. Then the solution $u(t,x)$ to the
equation
\begin{align}\label{lw}
\begin{cases}
u_{tt}-\Delta u=F,\\
u(t_0,x)=f(x),\ u_t(t_0,x)=g(x)
\end{cases}
\end{align}
satisfies the Duhamel's formula:
\begin{align}\label{duhamel}
u(t,x)=\cos(\sqrt{-\Delta}(t-t_0))
f+\frac{\sin((t-t_0)\sqrt{-\Delta})}{\sqrt{-\Delta}} g
+\int_{t_0}^{t}\frac{\sin((t-s)\sqrt{-\Delta})}{\sqrt{-\Delta}}F(s)
ds.
\end{align}
Moreover,
\begin{align*}
\|u\|_{\dot S^1(I)}\lsm \|f\|_{\dot H_x^1}+\|g\|_2 +\|F\|_{\dot
N^1(I)}.
\end{align*}
\end{lem}

\begin{rem}\label{rem:h1-adm} Let $d\ge 6$ and $(q,r)$ admissible.
Let $\tilde r$ be such that
\begin{align}\label{h1-admis}
\frac 1q+\frac d{\tilde r}=\frac{d-2}2.
\end{align}
Then by Sobolev embedding we have
$$
\|f\|_{L_x^{\tilde r}}\lsm \|f\|_{\dot W_x^{1-\beta(r),r}}.
$$
Therefore
$$
\|u\|_{L_t^q L_x^{\tilde r}(I\times\R^d)}\lsm \|u\|_{\dot S^1(I)}.
$$
For example, we can take $(q,\tilde r)=(\infty, \frac{2d}{d-2})$.
Other examples will be used in this paper without explicitly mentioning this
embedding.
\end{rem}

\subsection{Space $H^{m,m}$ and its basic properties.}
Let $m>0$ be an integer. Define $H^{m,m}$ to be the Banach space
with the norm
$$
\|f\|_{H^{m,m}}=\sum_{0\le j\le m}\|\langle x\rangle^{m-j}\nabla ^j
f\|_2.
$$
We collect several useful lemmas.

\begin{lem}[Estimates of linear solutions]\label{lin-hmm}
Let $m$ be a positive integer. Let $u$ be the solution of
\begin{align*}
\begin{cases}
u_{tt}-\Delta u=0,\\
u(0,x)=f(x),\ u_t(0,x)=g(x).
\end{cases}
\end{align*}
Then there exists a $m$-dependent constant $C>0$ such that
\begin{align}
\|u_t(t)\|_{H^{m,m}}+\|\nabla u(t)\|_{H^{m,m}}&\le (\|\nabla
f\|_{H^{m,m}}+\|g\|_{H^{m,m}})e^{Ct}.\label{lin-hmm1}\\
\|u(t)\|_{H^{m,m}}&\le (\|\nabla f\|_{H^{m,m}}+\|g\|_{H^{m,m}}+\|
f\|_{H^{m,m}})e^{Ct}.\label{lin-hmm2}
\end{align}
\end{lem}
\begin{proof} \eqref{lin-hmm1} follows directly from the standard
energy method. The second one \eqref{lin-hmm2} is a consequence of
\eqref{lin-hmm1} and the Fundamental Theorem of Calculus.
\end{proof}

\begin{lem}\label{sigma-hmm}
Let $t_0>0$, $\alpha>0$. Let $\Sigma_{t_0} $ be the Banach space
with the norm
\begin{align*}
\|u\|_{\Sigma_{t_0}}=\sup_{t\ge t_0} e^{\alpha t}\|u(t)\|_{H^{m,m}}.
\end{align*}
Then there exists a $m$-dependent constant $C$ such that
\begin{align}
\biggl\|\int_t^{\infty}\frac{\sin((t-\tau)\sqrt{-\Delta})}{\sqrt{-\Delta}}
F(\tau)d\tau\biggr\|_{\Sigma_{t_0}}\le \frac
1{\alpha-C}\|F\|_{\Sigma_{t_0}}.\label{est-int-hmm}
\end{align}
\end{lem}
\begin{proof}
Using Duhamel's formula, Lemma \ref{lin-hmm} gives
that
\begin{align*}
\biggl\|\frac{\sin(t\sqrt{-\Delta})}{\sqrt{-\Delta}}
g\biggr\|_{H^{m,m}}\le e^{C|t|}\|g\|_{H^{m,m}}.
\end{align*}
Applying this and Minkowski inequality we have
\begin{align*}
\biggl\|\int_t^{\infty}
\frac{\sin((t-\tau)\sqrt{-\Delta})}{\sqrt{-\Delta}}
F(\tau)d\tau\biggr\|_{H^{m,m}}&\le
\int_t^{\infty}\biggl\|\frac{\sin((t-\tau)\sqrt{-\Delta})}{\sqrt{-\Delta}}
F(\tau)\biggr\|_{H^{m,m}}d\tau\\
&\le \int_t^{\infty} e^{C|t-\tau|}\|F(\tau)\|_{H^{m,m}}d\tau\\
&\le \int_t^{\infty} e^{C(\tau-t)}
e^{-\alpha\tau}\|F\|_{\Sigma_{t_0}} d\tau\\
&\le \frac{e^{-\alpha t}}{\alpha-C}\|F\|_{\Sigma_{t_0}},
\end{align*}
which gives immediately \eqref{est-int-hmm}.
\end{proof}

We record several useful lemmas from \cite{lz:schrodinger}. The
proof can be found in \cite{lz:schrodinger}.

\begin{lem}[Embedding in $H^{m,m}$] \label{lem_embed_208}
Let $k_1$, $k_2$ be non-negative integers, then for any $m\ge k_1
+k_2 + \frac d2+1$, we have
\begin{align*}
 \| \langle x \rangle^{k_1} \nabla^{k_2} f\|_{\infty}
\lesssim \|f\|_{H^{m,m}},
\end{align*}
where the implicit constant depends only on $k_1$, $k_2$.
\end{lem}

\begin{lem}[Bilinear estimate in $H^{m,m}$]\label{bilinear}  We have
\begin{equation}
\|f g\|_{H^{m,m}}\lesssim \|f\|_{W^{m,\infty}}\|g\|_{H^{m,m}},
\end{equation}
with the implicit constant depending only on $m$.
\end{lem}

\begin{lem}\label{multi}
Let $C>0$, $j\ge 2$ and $m\ge \frac d2+1+\frac{Cj}{j-1}$, then
$$
\|\langle x\rangle^{Cj}h^j\|_{H^{m,m}}\lesssim j^m
\|h\|_{H^{m,m}}^j,
$$
where the implicit constant depends only on $m$.
\end{lem}

\subsection{Derivation of the perturbation equation near $W$.}
Let $u$ be the solution to the equation in \eqref{nlw}. Let $v=u-W$,
then $v$ satisfies the equation
\begin{align}\label{dif-eq}
v_{tt}-\Delta v=|v+W|^{p_c-1}(v+W)-W^{p_c}.
\end{align}
Let $R(v)=|v+W|^{p_c-1}(v+W)-p_cW^{p_c-1}v-W^{p_c}$ and $\mathcal L
v=(-\Delta-p_cW^{p_c-1})v$, then \eqref{dif-eq} can be written
equivalently as
\begin{align}
v_{tt}-\Delta v&=R(v)+p_c W^{p_c-1}v,\label{rv-eq}\\
v_{tt}+\mathcal L v&=R(v).\label{lv-eq}
\end{align}
We record the following spectral properties of $\mathcal L$ from
\cite{duck-merle:wave}.
\begin{lem}[Spectral property] \label{spe-l}
The operator $\mathcal L$ has no positive eigenvalue and a unique
negative eigenvalue $-e_0^2$ with the corresponding eigenfunction
$\mathcal Y\in \mathcal S(\R^d)$, i.e.
\begin{align*}
\mathcal L\mathcal Y=-e_0^2 \mathcal Y.
\end{align*}
For convenience of notations, we will assume $e_0>0$.
\end{lem}

\section{The existence of $W^-$, $W^+$.}

As in \cite{duck-merle:wave}, the threshold solutions $W^-$, $W^+$
are constructed as the limit of a sequence of near solutions
$W_k^a(t,x)$. On the other hand, the asymptotic behaviors of $W^{-}$ and $W^{+}$ are quite different
in the negative time direction (see Remark \ref{rem_revadd}). We need the following result:
\begin{lem}[\cite{duck-merle:wave}] \label{lemma:wka}
Let $a\in\R$ and let $e_0>0$ be the same as in Lemma \ref{spe-l}. There exist functions
$\{\Phi_j^a\}_{j\ge 1}$ in $\mathcal S(\R^d)$ such that
$\Phi_1^a=a\mathcal Y$(see Lemma \ref{spe-l} for the definition of
$\mathcal Y$) and if
$$
W_k^a(t,x)=W(x)+\sum_{j=1}^k e^{-je_0t}\Phi_j^a (x),
$$
then as $t\to\infty$,
\begin{align}\label{near}
\eps_k^a:=(\partial_{tt} -\Delta)
W_k^a-|W_k^a|^{p_c-1}W_k^a=O(e^{-(k+1)e_0t}),\mbox{ in } \mathcal
S(\R^d).
\end{align}
 More precisely, $\forall J, M \ge 0$, $J,M$ are integers, there exists
a constant $C_{J,M}$ such that
$$
\langle x\rangle^M|\nabla^J\eps_k^a(t,x)|\le C_{J,M}e^{-(k+1)e_0t},
$$
for all $t$ sufficiently large.
\end{lem}

\begin{rem}\label{rm:provk}
Since all $\Phi_j$ are Schwartz functions, we have the following
properties for the difference
\begin{align}\label{vk-defn}
v_k=W_k^a-W=\sum_{j=1}^k e^{-je_0t}\Phi_j^a(x).
\end{align}
For any $j,l\ge 0$ and $1\le p\le \infty$, there exists $C_{k,j,l}>0$ such that
\begin{align}
|\langle x\rangle^j \nabla^l v_{k}(t,x)|&\le C_{k,j,l}e^{- {e_0}
t}.\label{proper-vk}\\
\|\langle \cdot \rangle^j \nabla^l v_k(t)\|_{p}&\le C_{k,j,l}
e^{-e_0 t}.\label{lp-vk}
\end{align}
\end{rem}

\vspace{0.2cm}

Next we show that there exists a unique genuine solution $W^a(t,x)$ of
 \eqref{nlw} which can be approximated by the above constructed near
 solutions $W_k^a(t,x)$. The existence and uniqueness of the solution $W^a$ is
  transformed to that of $h:=W^a-W_k^a$ which satisfies the equation
\begin{equation}\label{equu-h}
\partial_{tt} h-\Delta h=p_cW^{p_c-1}h+R(h+v_k)-R(v_k)-\eps_k^a.
\end{equation}
Like for the Schr\"odinger equation \cite{lz:schrodinger}, we will
construct the solution to \eqref{equu-h} by using fixed point
argument in the weighted Sobolev space $H^{m,m}$. The reason is that
the nonlinearity in \eqref{equu-h} can be shown to be Lipschitz
continuous while the space used in \cite{duck-merle:wave} does not
work for higher dimensions.

\begin{prop}\label{ex-hmm}
Let $a\in \R $. Let $\mathcal Y$ and $W_k^a=W_k^a(t,x)$ be the same
as in Lemma \ref{lemma:wka}. Assume $m\ge 3d$ is fixed. Then there
exists $k_0>0$ and a unique solution $W^a(t,x)$ for the equation in
\eqref{nlw} which satisfies the following: for any $k\ge k_0$, there
exists $t_k\ge 0$ such that $\forall \ t\ge t_k$,
\begin{equation}\label{goal}
\|W^a(t)-W_k^a(t)\|_{H^{m,m}}\le e^{-(k+\frac 12)e_0 t}.
\end{equation}
Moreover, we have
\begin{equation}\label{extra-fir}
\|\nabla
(W^a(t)-W_k^a)\|_{H^{m,m}}+\|\partial_t(W^a(t)-W_k^a(t)\|_{H^{m,m}}\le
e^{-(k+\frac 12)e_0t}.
\end{equation}
\end{prop}

\begin{proof} Let $h=W^a-W_k^a$, then $W^a$ is the solution of
\eqref{nlw} as long as $h$ is a solution of the equation
\eqref{equu-h}. By Duhamel's formula, the existence of the solution
to \eqref{equu-h} which satisfies the decay condition \eqref{goal},
\eqref{extra-fir} is transformed into the existence of the solution
to the following integral equation for large time $t$,
\begin{align}\label{integral-eq}
h(t)&=-\int_t^{\infty}\frac{\sin((t-\tau)\sqrt{-\Delta})}{\sqrt{-\Delta}}
(p_cW^{p_c-1}h+R(h+v_k)-R(v_k)-\eps_k^a)(\tau) d\tau\\
&:=\Phi(h)(t)
\end{align}
Define the space $\Sigma_{t_k}$ to be the space with the norm
\begin{align*}
\|f\|_{\Sigma_{t_k}}=\sup_{t\ge t_k}e^{\alpha t}\|f(t)\|_{H^{m,m}},
\ \alpha=(k+\frac 12)e_0,
\end{align*}
and the unit ball
\begin{align*}
B_k=\{f=f(t,x);\ \|f\|_{\Sigma_{t_k}}\le 1\}.
\end{align*}
We will show $\Phi$ is a contraction on $B_k$. Taking $h\in B_k$, we
compute the $H^{m,m}$ norm of $\Phi(h)(t)$ as follows
\begin{align}
\|\Phi(h)(t)\|_{H^{m,m}} &\le p_c
\int_t^{\infty}\biggl\|\frac{\sin((t-\tau)\sqrt{-\Delta})}{\sqrt{-\Delta}}W^{p_c-1}
h(\tau)\biggr\|_{H^{m,m}} d\tau\label{256}\\
&+\int_t^{\infty}\biggl\|\frac{\sin((t-\tau)\sqrt{-\Delta})}{\sqrt{-\Delta}}(R(h+v_k)-R(v_k))
(\tau)\biggr\|_{H^{m,m}} d\tau\label{257}\\
&+\int_t^{\infty}\biggl\|\frac{\sin((t-\tau)\sqrt{-\Delta})}{\sqrt{-\Delta}}\eps_k^a
(\tau)\biggr\|_{H^{m,m}} d\tau\label{258}.
\end{align}
To estimate \eqref{256}, we use Lemma \ref{lin-hmm}, Lemma
\ref{bilinear} to get
\begin{align}
\eqref{256}&\lsm \int_t^{\infty} e^{C|t-\tau|}\|W^{p_c-1}h(\tau)\|_{H^{m,m}}d\tau\label{305}\\
&\lesssim \int_t^{\infty}
e^{C|t-\tau|}\|W^{p_c-1}\|_{W^{m,\infty}}\|h(\tau)\|_{H^{m,m}}d\tau \notag \\
&\lesssim \int_t^{\infty} e^{C(\tau-t)}e^{-\alpha
\tau}\|h\|_{\Sigma_{t_k}}
d\tau\notag\\
&\lesssim e^{-Ct}\|h\|_{\Sigma_{t_k}}\int_{t}^{\infty}e^{-(\alpha-C)\tau}d\tau \notag\\
&\lesssim  \frac 1{\alpha-C} e^{-\alpha t}\|h\|_{\Sigma_{t_k}}.\notag
\end{align}
Since $\alpha=(k+\frac 12)e_0$, by taking $k_0$ sufficient
large, we have
\begin{equation}
\eqref{256}\le \frac 1{100} e^{-\alpha
t}\|h\|_{\Sigma_{t_k}}\le\frac 1{100} e^{-\alpha
t}\label{linear-est}
\end{equation}
for all $k\ge k_0$.

Now we deal with \eqref{258}. Note that by Lemma \ref{lemma:wka},
$\eps_k^a(t)=O(e^{-(k+1)e_0t})$ in $\mathcal S(\R^d)$. This implies
\begin{align*}
\|\eps_k^a(t)\|_{H^{m,m}}\le C_k e^{-(k+1)e_0t}.
\end{align*}
Thus,
\begin{align}
\eqref{258}&\le \int_t^{\infty}
e^{C|t-\tau|}\|\eps_k^a(\tau)\|_{H^{m,m}}d\tau\label{eps-est}\\
&\le C_k \int_t^{\infty} e^{C|t-\tau|}e^{-(k+1)e_0\tau}ds\tau\notag\\
&\le  C_k\frac {e^{-\frac 12 e_0 t}} {(k+1)e_0-C}e^{-(k+\frac
12)e_0t}\le \frac 1{100}e^{-\alpha t},\notag
\end{align}
if $t\ge t_k$ and $t_k$ is sufficiently large.

It remains to estimate \eqref{257}. The reason that we can take $m$
derivatives is that both $v_k$ and $h$  are small compared to $W$.
Indeed by Remark \ref{rm:provk}, we have
\begin{equation}\label{vsmall}
|v_k(t,x)|<\frac 12 W(x), \quad \forall \ t\ge t_k, \ x\in \R^d.
\end{equation}
Moreover, since $h\in \Sigma_{t_k}$ and $m\ge 3d$, by Lemma \ref{lem_embed_208} we have
\begin{align*}
 &\|\langle x\rangle ^{d-2} h(t)\|_{\infty} \\
\lesssim & \|h(t)\|_{H^{m,m}} \le e^{-\alpha t}\|h\|_{\Sigma_{t_k}}.
\end{align*}
As a consequence, we have
\begin{equation}\label{hsmall}
|h(t,x)|\lesssim e^{-\alpha t} \langle
x\rangle^{-(d-2)}\|h\|_{\Sigma_{t_k}}\le \frac 14 W(x).
\end{equation}
Using \eqref{vsmall} and \eqref{hsmall} together with the expansion
for the real analytic function $P(s)=|1+s|^{p_c-1}(1+s)$ for $|s|\le
\frac 34$ which takes the form
\begin{align}\label{exp-pz}
P(s)=1+p_c s+\sum_{j\ge 2} a_j s^j, \ |a_j|\lsm 1,
\end{align}
  we write
\begin{align}
&R(v_k+h)-R(v_k)\label{1form}\\
&=W^{p_c}\biggl(|1+\frac{v_k+h}W|^{p_c-1}(1+\frac{v_k+h}{W})
-|1+\frac{v_k}W|^{p_c-1}(1+\frac{v_k}W)-p_c\frac
hW\biggr)\notag\\
&=\sum_{\substack {j\ge 2\\1\le i\le j}}a_jC_{i,j}
W^{p_c-j}v_k^{j-i}h^i,\notag
\end{align}
where $C_{i,j}=\frac{j!}{i!(j-i)!}\le 2^j$. By triangle inequality and
Lemma \ref{bilinear}, we estimate the $H^{m,m}$-norm of \eqref{1form} as follows:
\begin{align*}
&\|R(v_k+h)(t)-R(v_k)(t)\|_{H^{m,m}}\\
&\lsm \sum_{\substack{j\ge 2\\1\le
i\le j}}2^j\|W^{p_c-j}v_k(t)^{j-i}h(t)^i\|_{H^{m,m}}\\
&\lsm \sum_{j\ge
2}2^j\|W^{-j}v_k(t)^{j-1}h(t)\|_{H^{m,m}}+\sum_{\substack{j\ge
2\\2\le i\le j}}2^j\|(W^{-1}v_k(t))^{j-i}W^{-i}h(t)^i\|_{H^{m,m}}\\
&\lsm \sum_{j\ge
2}2^j\|W^{-j}v_k(t)^{j-1}\|_{W^{m,\infty}}\|h(t)\|_{H^{m,m}}+
\sum_{\substack{j\ge 2\\2\le i\le j}}2^j
\|(W^{-1}v_k(t))^{j-i}\|_{W^{m,\infty}}\|W^{-i}h^i\|_{H^{m,m}}.
\end{align*}
Here we have split the sum in the index $i$ because our Lemma
\ref{multi} requires $i\ge 2$.
Now by Remark \ref{rm:provk} and Lemma \ref{multi} we have
\begin{align*}
\|W^{-i}h^i(t)\|_{H^{m,m}}&\lsm i^m\|h(t)\|_{H^{m,m}}^i,\\
\|W^{-j}v_k(t)^{j-1}\|_{W^{m,\infty}}&\lsm j^m
C_{k}e^{-(j-1)e_0t},\\
\|(W^{-1}v_k(t))^{j-i}\|_{W^{m,\infty}}&\lsm j^m
C_{k}e^{-(j-i)e_0t}.
\end{align*}
Note moreover that
\begin{align*}
\|h(t)\|_{H^{m,m}}\le e^{-\alpha t}\|h\|_{\Sigma_{t_k}},
\end{align*}
we have
\begin{align*}
\|R(v_k+h)(t)-R(v_k)(t)\|_{H^{m,m}}&\lsm \sum_{\substack{j\ge
2\\1\le i\le j}}2^j j^{2m}C_{k}e^{-(j-i)e_0t}(e^{-\alpha
t}\|h\|_{\Sigma_{t_k}})^i\\
&\lsm e^{-\alpha t}\|h\|_{\Sigma_{t_k}}\sum_{\substack{j\ge 2\\1\le
i\le j}}2^j j^{2m}C_{k}e^{(-\alpha(i-1)-(j-i)e_0)t}\\
&\lsm e^{-\alpha t}\|h\|_{\Sigma_{t_k}}\sum_{\substack{j\ge 2\\1\le
i\le j}}2^j j^{2m}C_{k}e^{-(j-1)e_0t_k},
\end{align*}
where in the second inequality we have dropped the term
$\|h\|_{\Sigma_{t_k}}^{i-1}$ since $\|h\|_{\Sigma_{t_k}} \le 1$.
Obviously the last series can be made arbitrarily small if we
choose $t_k$ sufficiently large. This gives us
\begin{align*}
\|R(v_k+h)(t)-R(v_k)(t)\|_{H^{m,m}}\le \frac 1{100}e^{-\alpha t}
\|h\|_{\Sigma_{t_k}}.
\end{align*}
Applying this estimate and Lemma \ref{lin-hmm} we have
\begin{align}
\eqref{257}&\le \int_t^{\infty}
e^{C|t-\tau|}\|R(h+v_k)(\tau)-R(v_k)(\tau)\|_{H^{m,m}} d\tau\label{rest}\\
&\le \frac 1{100}\int_t^{\infty}
e^{C|t-\tau|}e^{-\alpha\tau}\|h\|_{\Sigma_{t_k}} d\tau\notag\\
&\le \frac 1{100} e^{-\alpha t} \|h\|_{\Sigma_{t_k}}.\notag
\end{align}
Collecting the estimates \eqref{linear-est}, \eqref{eps-est} and
\eqref{rest} we obtain
\begin{align*}
\|\Phi(h)(t)\|_{H^{m,m}}\le \frac 1{10} e^{-\alpha
t}\|h\|_{\Sigma_{t_k}}
\end{align*}
for all $t\ge t_k$, $k\ge k_0$. Therefore
\begin{align*}
\|\Phi(h)(t)\|_{\Sigma_{t_k}}\le \frac 1{10},
\end{align*}
which shows that $\Phi$ maps $B_k$ to itself. To show $\Phi$ is a
contraction, we choose $h_1$, $h_2\in B_k$ and estimate
\begin{align}
\|\Phi(h_1)(t)&-\Phi(h_2)(t)\|_{H^{m,m}}\notag\\
&\le p_c
\int_t^{\infty}\biggl\|\frac{\sin((t-\tau)\sqrt{-\Delta})}{\sqrt{-\Delta}}
W^{p_c-1}(h_1-h_2)(\tau)\biggr\| d\tau\label{diff-lin}\\
&+\int_t^{\infty}\biggl\|\frac{\sin((t-\tau)\sqrt{-\Delta})}{\sqrt{-\Delta}}
(R(h_1+v_k)(\tau)-R(h_2+v_k)(\tau))\biggr\| d\tau.\label{diff-r}
\end{align}
The estimate of \eqref{diff-lin} is the same as \eqref{256}, we
get
\begin{align*}
\eqref{diff-lin}\le \frac 1{100} e^{-\alpha
t}\|h_1-h_2\|_{\Sigma_{t_k}},\ \ \forall \ k\ge k_0.
\end{align*}
To estimate \eqref{diff-r}, we use \eqref{exp-pz} to write
\begin{align*}
R(h_1+v_k)-R(h_2+v_k)
&=W^{p_c} \Bigl( |1+\frac{h_1+v_k}W|^{p_c-1}(1+\frac{h_1+v_k}W) \Bigr.\\
&\qquad\qquad \Bigl. -|1+\frac{h_2+v_k}W|^{p_c-1}(1+\frac{h_2+v_k}W)-p_cW^{p_c-1}(h_1-h_2) \Bigr)\\
&=\sum_{j\ge 2} a_j W^{p_c-j}((h_1+v_k)^j-(h_2+v_k)^j)\\
&=\sum_{\substack{j\ge 2\\1\le i\le j}}a_jC_{i,j}
(h_2+v_k)^{j-i}(h_1-h_2)^i.
\end{align*}
With minor changes, this term can be treated in the same manner as
\eqref{1form}, so we have
\begin{align*}
\eqref{diff-r}\le \frac 1{100} e^{-\alpha
t}\|h_1-h_2\|_{\Sigma_{t_k}},\ \forall\ k\ge k_0,\ t\ge t_k.
\end{align*}
Therefore,
\begin{align*}
\|\Phi(h_1)(t)-\Phi(h_2)(t)\|_{H^{m,m}}&\le \frac 1{10} e^{-\alpha
t}\|h_1-h_2\|_{\Sigma_{t_k}}, \ \forall\ k\ge k_0,\ t\ge t_k.\\
\|\Phi(h_1)-\Phi(h_2)\|_{\Sigma_{t_k}}&\le \frac
1{10}\|h_1-h_2\|_{\Sigma_{t_k}}.
\end{align*}
This proves the map $\Phi$ is a contraction on $B_k$, hence there
exists a unique solution $h$ to the equation \eqref{equu-h} such
that
\begin{align}
\|h(t)\|_{H^{m,m}}\le e^{-\alpha t}, \ \forall \ t\ge t_k,\ k\ge
k_0.\label{h-estt}
\end{align}
Note $h=W^a-W_k^a$, this means that for any $k\ge k_0$, there exists
a unique solution $W^a(t)$ to the equation \eqref{nlw} on
$[t_k,\infty)$ such that
\begin{align*}
\|W^a(t)-W_k^a(t)\|_{H^{m,m}}\le e^{-(k+\frac 12)e_0t}.
\end{align*}
We need to show that $W^a(t,x)$ is independent of $k$. Indeed, let
$k_1<k_2$ and $W^a$, $\widetilde {W^a}$ be the corresponding
solutions such that
\begin{align*}
\|W^a(t)-W_{k_1}^a(t)\|_{H^{m,m}}\le e^{-(k_1+\frac 12)e_0 t},\ \forall\ t\ge t_{k_1},\\
\|\widetilde{W^a}(t)-W_{k_2}^a(t)\|_{H^{m,m}}\le e^{-(k_2+\frac
12)e_0 t},\ \forall \ t\ge t_{k_2}.
\end{align*}
Without loss of generality we also assume $t_{k_1}\le t_{k_2}$, then
the triangle inequality gives that
\begin{align*}
\|\widetilde{W^a}(t)-W_{k_1}^a(t)\|_{H^{m,m}}&\le \|\widetilde{
W^a}(t)-W_{k_2}^a(t)\|_{H^{m,m}}+\|\sum_{k_1< j\le k_2}e^{-je_0t}
\Phi_j\|_{H^{m,m}}\\
&\le e^{-(k_1+\frac 12)e_0 t},\  \ \forall\ t\ge t_{k_2}.
\end{align*}
Therefore $W^a(t)=\widetilde{ W^a}(t)$ on $[t_{k_2},\infty)$ and we
conclude $W^a\equiv\widetilde{ W^a}$ by uniqueness of solutions to
\eqref{nlw}. This shows that $W^a$ does not depend on $k$.

We finally verify \eqref{extra-fir}. Using Duhamel's formula
\eqref{integral-eq} and Lemma \ref{lin-hmm}, we have
\begin{align*}
&\|h_t(t)\|_{H^{m,m}}+\|\nabla h(t)\|_{H^{m,m}}\\
&\le \int_t^{\infty}\biggl
\|\partial_t\frac{\sin((t-\tau)\sqrt{-\Delta})}{\sqrt{-\Delta}}
(p_cW^{p_c-1}h(\tau)+R(h+v_k)(\tau)-R(v_k)(\tau)-\eps_k^a(\tau))\biggr\|_{H^{m,m}}
d\tau\\
&+ \int_t^{\infty}\biggl
\|\nabla\frac{\sin((t-\tau)\sqrt{-\Delta})}{\sqrt{-\Delta}}
(p_cW^{p_c-1}h(\tau)+R(h+v_k)(\tau)-R(v_k)(\tau)-\eps_k^a(\tau))\biggr\|_{H^{m,m}}
d\tau\\
&\le \int_t^{\infty} e^{C|t-\tau|}\|p_c W^{p_c-1}
h(\tau)+R(h+v_k)(\tau)-R(v_k)(\tau)-\eps_k^a(\tau)\|_{H^{m,m}}d\tau\\
&\le p_c\int_t^{\infty} e^{C|t-\tau|}\|W^{p_c-1}h(\tau)\|_{H^{m,m}}
d\tau\\
&\quad+\int_t^{\infty}
e^{C|t-\tau|}\|R(h+v_k)(\tau)-R(v_k)(\tau)\|_{H^{m,m}}d\tau\\
&\quad+\int_t^{\infty}e^{C|t-\tau|}\|\eps_k^a(\tau)\|_{H^{m,m}}
d\tau.
\end{align*}
These terms have been estimated before (see \eqref{256},
\eqref{257}, \eqref{258}). With the condition \eqref{h-estt}, we
have
\begin{align*}
\|h_t(t)\|_{H^{m,m}}+\|\nabla h(t)\|_{H^{m,m}}&\le e^{-(k+\frac
12)e_0 t}.
\end{align*}
The Proposition is proved.
\end{proof}

 \begin{cor}\label{prop-wa}
Let $k_0$ be the same as in Proposition \ref{ex-hmm}. Let
$v_{k_0}=\sum_{j=1}^{k_0}e^{-je_0 t}\phi_j(x)$ and $w^a=W^a-W$. Then
there exists $t_0>0$ such that for all $t\ge t_0$ and all $2\le
p\le\infty$
\begin{align}
\|\langle x\rangle^{l_1}\nabla^{l_2}(w^a(t)-v_{k_0}(t))\|_{L_x^p}
\lesssim e^{-(k_0+1)e_0 t} \le
e^{-(k_0+\frac 12)
e_0 t},\label{norm-wa}\\
\|\langle x\rangle^{l_1}\nabla^{l_2}w^a(t)\|_{L_x^p}
\lesssim e^{-e_0 t}
\le e^{-\frac 12
e_0 t},
\end{align}
as long as $l_1+l_2+\frac d2+1\le m$. In particular,
\begin{align*}
\|w^a-v_{k_0}\|_{\dot S^1([t,\infty))}\le e^{-k_0 e_0 t},\ \
\|w^a\|_{\dot S^1([t,\infty))}\le e^{-\frac 12 e_0 t}.
\end{align*}
\end{cor}
\begin{proof}
By Proposition \ref{ex-hmm},
\begin{align*}
\|w^a(t)-v_{k_0}(t)\|_{H^{m,m}}=\|W^a(t)-W_{k_0}^a(t)\|_{H^{m,m}}&\le
e^{-(k_0+\frac 12) e_0 t},\ \ \forall\ t\ge t_{k_0}.
\end{align*}
An application of Sobolev embedding (see Lemma \ref{lem_embed_208}) and interpolation yields that for any $2\le p\le
\infty$ and $t$ sufficiently large
\begin{align*}
\|\langle x\rangle^{l_1}\nabla^{l_2}(w^a(t)-v_{k_0}(t))\|_{L_x^p}&\lsm
e^{-(k_0+\frac 12)e_0 t}\\
&\le \frac 12 e^{-k_0 e_0 t},
\end{align*}
as long as $m\ge l_1+l_2+\frac d2+1$. In particular, for any $2\le
r\le \frac{2(d-1)}{d-3}$,
\begin{align*}
&\quad \||\nabla|^{1-\beta(r)}(w^a(t)-v_{k_0}(t))\|_{L_x^r}+
\||\nabla|^{-\beta(r)}(w^a(t)-v_{k_0}(t)\|_{L_x^r}\\
&\lesssim \|\langle
\nabla\rangle(w^a(t)-v_{k_0}(t))\|_{L_x^r}+\|w^a(t)-v_{k_0}(t)\|_{L_x^{\frac{4rd}{r(d+1)+2d-2}}}\\
&\lsm e^{-(k_0+\frac 12)e_0 t}.
\end{align*}
In the last equality, we use the fact that
$\frac{4rd}{r(d+1)+2d-2}\ge 2$ for all $r\ge 2$. Integrating the time
variable over $[t,\infty)$, we obtain
\begin{align*}
\|w^a-v_{k_0}\|_{\dot S^1([t,\infty))}\lsm e^{-(k_0+\frac 12)e_0
t}\le \frac 12 e^{-k_0 e_0 t}.
\end{align*}
On the other hand, since $v_{k_0}$ is the combination of Schwartz
functions, we have by \eqref{lp-vk}
\begin{align*}
%\|\langle x\rangle^{l_1}\nabla^{l_2}v_{k_0}(t)\|_{L_x^p}\le
%C_{l_1,l_2} e^{-e_0 t}\le \frac 12 e^{-\frac{e_0}2 t}.\\
\|v_{k_0}\|_{\dot S^1([t,\infty))}\lsm e^{-e_0 t}\le \frac 12
e^{-\frac{e_0}2 t},
\end{align*}
for all $t$ sufficiently large. The estimates of $w^a$ then follow from triangle inequality.
\end{proof}

%Finally we remark that in the next section we shall show  for $a,b$
%such that $ab>0$, $W^a$ is just a time translation of $W^b$.  This
%will allow us to define $W^{\pm}$ as $W^{\pm 1}$ and to classify the
%solutions with threshold energy.

\begin{rem} \label{rem_revadd}
From the construction of $W^{\pm}(t)$, it is clear that they both
approaches to the ground state $W$ exponentially fast as $t \to
+\infty$. For the behavior of $W^{\pm}$ in negative time direction,
we can apply the same argument in \cite{duck-merle} (see proof
of Theorem 1, Proposition 2.8, Proposition 3.1 and Subsection 6.4 for instance) to conclude that $W^{-}$ scatters
when $t\to -\infty$ and $W^{+}$ blows up at finite time.
%To get the
%blowup of $W^{+}$, we need the crucial property $W^{+}\in L_x^2$
%which is now available as we are in dimensions $d\ge 6$.
\end{rem}

\section{Two useful estimates}

First we show that $R(v)$ is superlinear in $v$, we have

\begin{lem}[Super-linearity of $R(v)$]\label{super-rv}
Let $I$ be a time slab.
%Let $\dot N^1(I)$, $\dot S^1(I)$ be defined in the introduction.
We have
\begin{align}
\|R(v)(t)\|_{L_x^{\frac{2d}{d+2}}}&\lsm
\|v(t)\|^{p_c}_{\dot H_x^1},\label{sp-1}\\
 \|R(v)\|_{\dot N^1(I)}&\lsm
\|v\|_{\dot S^1(I)}^{\frac{d+3}{d+1}}+\|v\|_{\dot
S^1(I)}^{p_c}.\label{sp-2}
\end{align}
\end{lem}
\begin{proof} By the definition of $R(v)$, we write
\begin{align*}
R(v)&=|v+W|^{p_c-1}(v+W)-p_c W^{p_c-1}v-W^{p_c}-|v|^{p_c-1}v+|v|^{p_c-1}v\\
&=W^{p_c}J(\frac vW)+|v|^{p_c-1}v,
\end{align*}
where
\begin{align} \label{js-def}
J(s)=|1+s|^{p_c-1}(1+s)-p_c s-1-|s|^{p_c-1}s.
\end{align}
Note $J(s)$ is
differentiable and for $d\ge 6$,
\begin{align}\label{js}
J(s)\le \begin{cases} |s|,& |s|\ge\frac 12,\\ |s|^{p_c},& |s|<\frac
12.\end{cases}\qquad J'(s)\le\begin{cases} 1, & |s|^{p_c-1} \ge \frac
12,\\|s|,& |s|<\frac 12.\end{cases}
\end{align}
We first give a quick proof of \eqref{sp-1}. Since by \eqref{js},
$J(s)\lsm |s|^{p_c}$, we have
\begin{align*}
\|R(v)(t)\|_{L_x^{\frac{2d}{d+2}}}&\lsm\||v(t)|^{p_c}\|_{L_x^{\frac{2d}{d+2}}}\\
&\lsm \|v(t)\|_{L_x^{\frac{2d}{d-2}}}^{p_c}\lsm
\|v(t)\|_{H_x^1}^{p_c}.
\end{align*}

Now we compute the $\dot N^1$ norm of $R(v)$. In the following, all
spacetime norms are on $I\times \R^d$. We have
\begin{align}\label{air}
\|R(v)\|_{\dot N^1(I)}\le \||\nabla|^{\frac 12}(W^{p_c}J(\frac
vW))\|_{L_{t,x}^{\frac{2(d+1)}{d+3}}}+\||\nabla|^{\frac
12}(|v|^{p_c-1}v)\|_{L_{t,x}^{\frac{2(d+1)}{d+3}}}.
\end{align}
For the second term, we use Lemma \ref{chr} and H\"older inequality
to get
\begin{align}
\||\nabla|^{\frac
12}(|v|^{p_c-1}v)\|_{L_{t,x}^{\frac{2(d+1)}{d+3}}}&\le
\||v|^{p_c-1}\|_{L_{t,x}^{\frac{d+1}2}}\||\nabla|^{\frac 12}
v\|_{L_{t,x}^{\frac{2(d+1)}{d-1}}}\label{bir}\\
&\lsm \|v\|_{L_{t,x}^{\frac{2(d+1)}{d-2}}}^{p_c-1}\||\nabla|^{\frac
12}v\|_{L_{t,x}^{\frac{2(d+1)}{d-1}}}\notag\\
&\lsm \|v\|_{\dot S^1(I)}^{p_c},\notag
\end{align}
which is good for us. For the first term on the RHS of \eqref{air},
we follow the idea in \cite{duck-merle:wave} and cut it into dyadic
pieces. To this end, we introduce a smooth cutoff function $\phi(x)$
which satisfies: $\phi(x)=1$ when $|x|\le 1$ and $\phi(x)=0$ when
$|x|>2$. Denote
\begin{align*}
\psi_k(x)=\begin{cases} \phi(x),& k=0,\\ \phi(\frac
x{2^{k}})-\phi(\frac x{2^{k-1}}),& k\ge 1.\end{cases}
\end{align*}
Then we verify that
\begin{align*}
supp\ \psi_k(x)=\{x:\ |x|\sim 2^k\};\ & \ supp\ \psi_0(x)=\{x:\
|x|\lsm
1\}.\\
\sum_{k\ge 0} \psi_k(x)&=1,\ \forall\ x\in\R^d.
\end{align*}
We also introduce a "fat" cutoff function $\tilde\psi_k(x)$ which
equals one the support of $\psi_k$. By rescaling, it is
straightforward to verify that: for any $1 \le  p \le \infty$ and
$\beta\in \R$
\begin{align}\label{prop-ps}
\|\langle x\rangle^{\beta}\psi_k\|_p\lsm 2^{k(\frac dp+\beta)},\
\||\nabla|^{\frac 12}(\langle x\rangle^{\beta}\psi_k)\|_p\lsm
2^{k(\frac dp+\beta-\frac 12)}.
\end{align}
The same estimates hold if we replace $\psi_k$ by $\tilde\psi_k$.

Now we use Lemma \ref{chr} to estimate the first term on the RHS in
\eqref{air} as follows:
\begin{align}
&\||\nabla|^{\frac 12}(W^{p_c}J(\frac
vW))\|_{L_{t,x}^{\frac{2(d+1)}{d+3}}}\notag\\
&\le \sum_{k\ge 0}\||\nabla|^{\frac12}(W^{p_c}\psi_k J(\tilde
\psi_k\frac vW))\|_{L_{t,x}^{\frac{2(d+1)}{d+3}}}\notag\\
&\lsm \sum_{k\ge 0}\||\nabla|^{\frac
12}(W^{p_c}\psi_k)\|_{p_1}\|J(\tilde \psi_k\frac
vW)\|_{L_t^{\frac{2(d+1)}{d+3}}L_x^{p_2}}\label{sum1}\\
& \quad+\sum_{k\ge 0}\|W^{p_c}\psi_k\|_{q_1}\||\nabla|^{\frac
12}J(\tilde \psi_k \frac
vW)\|_{L_t^{\frac{2(d+1)}{d+3}}L_x^{q_2}},\label{sum2}
\end{align}
where $1< p_1,p_2,q_1,q_2<\infty$ are chosen such that
\begin{align*}
\frac 1{p_1}+\frac 1{p_2}=\frac 1{q_1}+\frac
1{q_2}=\frac{d+3}{2(d+1)}.
\end{align*}
We first deal with \eqref{sum1}. Choosing $p_1=\frac
{4d(d+1)} {3(d+3)}$, $p_2=\frac{4d(d+1)}{(d+3)(2d-3)}$,
$p_3=\frac{4d(d+1)}{3(d+3)}$, $p_4=\frac{2d(d+1)}{d^2-9}$, we check
that
\begin{align*}
\frac 1{p_1}+\frac 1{p_2}=\frac{d+3}{2(d+1)},\ \frac 1{p_2}=\frac
1{p_3}+\frac 1{p_4}.
\end{align*}
Noting $W^{p_c}\lsm \langle x\rangle ^{-(d+2)}$, $|J(s)|\lsm
|s|^{\frac{d+3}{d+1}}$, using \eqref{prop-ps} we bound the summand
in \eqref{sum1} by
\begin{align*}
&2^{k(\frac d{p_1}-d-2-\frac12)}\|\tilde \psi_k
W^{-\frac{d+3}{d+1}}v^{\frac{d+3}{d+1}}\|_{L_t^{\frac{2(d+1)}{d+3}}L_x^{p_2}}\\
&\qquad\lsm 2^{k(\frac d{p_1}-d-\frac 52)}\|\tilde \psi_k
W^{-\frac{d+3}{d+1}}\|_{L_x^{p_3}}\|v\|^{\frac {d+3}{d+1}}_{L_t^2L_x^{\frac{p_4(d+3)}{d+1}}}\\
&\qquad\lsm 2^{k(\frac d{p_1}-d-\frac 52)}2^{k(\frac
d{p_3}+\frac{(d-2)(d+3)}{d+1})}\|v\|^{\frac {d+3}{d+1}}_{L_t^2L_x^{\frac{2d}{d-3}}}\\
&\qquad\lsm 2^{-k\frac{d+4}{d+1}}\|v\|_{\dot
S^1(I)}^{\frac{d+3}{d+1}}.
\end{align*}
Summing in $k$, we obtain
\begin{align}
\eqref{sum1}\lsm \|v\|_{\dot
S^1(I)}^{\frac{d+3}{d+1}}.\label{est-sum1}
\end{align}

We now deal with \eqref{sum2}. Let $q_1=\frac{d(d+1)}6$,
$q_2=\frac{2d(d+1)}{d^2+3d-12}$, $q_3=\frac{2(d+1)}{d-1}$,
$q_4=\frac{d(d+1)}{2(d-3)}$. We verify that
\begin{align*}
\frac 1{q_1}+\frac 1{q_2}=\frac{d+3}{2(d+1)},\ \ \frac 1{q_3}+\frac
1{q_4}=\frac 1{q_2}.
\end{align*}
Note by \eqref{js}, we have $|J'(s)|\lsm |s|^{\frac 4{d+1}}$. Using
\eqref{prop-ps} and Lemma \ref{chr} we estimate the summand in
\eqref{sum2} by
\begin{align*}
&\|W^{p_c}\psi_k\|_{L_x^{q_1}}\||\nabla|^{\frac 12} J(\tilde \psi_k
W^{-1}v)\|_{L_t^{\frac{2(d+1)}{d+3}}L_x^{q_2}}\\
&\lsm \|W^{p_c}\psi_k\|_{L_x^{q_1}}\|J'(\tilde \psi_k
W^{-1}v)\|_{L_t^{\frac{d+1}2}L_x^{q_4}}\||\nabla|^{\frac 12}(\tilde
\psi_k W^{-1}v)\|_{L_t^{\frac{2(d+1)}{d-1}}L_x^{q_3}}\\
&\lsm \|W^{p_c}\psi_k\|_{L_x^{q_1}}\|\tilde
\psi_kW^{-1}v\|_{L_t^2L_x^{\frac{2d}{d-3}}}^{\frac 4{d+1}} \\
&\qquad \cdot
\biggl(
\||\nabla|^{\frac12}v\|_{L_{t,x}^{\frac{2(d+1)}{d-1}}}\|\tilde\psi_k
W^{-1}\|_{L_x^{\infty}}+\||\nabla|^{\frac 12}(\tilde\psi_k
W^{-1} )\|_{L_x^{2d}}\|v\|_{L_t^{\frac{2(d+1)}{d-1}}L_x^{\frac{2d(d+1)}{d^2-2d-1}}}\biggr)\\
&\lsm \|W^{p_c}\tilde \psi_k\|_{L_x^{q_1}}\|\tilde\psi_k
W^{-1}\|_{L_x^{\infty}}^{\frac 4{d+1}}\|v\|_{\dot S^1(I)}^{1+\frac
4{d+1}}(\|\tilde \psi_k W^{-1}\|_{L_x^{\infty}}+\||\nabla|^{\frac
12}(\tilde\psi_k W^{-1})\|_{L_x^{2d}})\\
&\lsm 2^{k(\frac d{q_1}-d-2+\frac{4(d-2)}{d+1}+d-2)}\|v\|_{\dot
S^1(I)}^{1+\frac 4{d+1}}\\
&\lsm 2^{-\frac {6k}{d+1}}\|v\|_{\dot S^1(I)}^{1+\frac 4{d+1}}.
\end{align*}
Summing in $k$ we obtain
\begin{align}
\eqref{sum2}\lsm \|v\|_{\dot
S^1(I)}^{\frac{d+5}{d+1}}.\label{est-sum2}
\end{align}
Collecting the estimates \eqref{est-sum1}, \eqref{est-sum2} we
obtain
\begin{align}
\||\nabla|^{\frac 12}(W^{p_c}J(\frac
vW))\|_{L_{t,x}^{\frac{2(d+1)}{d+3}}}\lsm \|v\|_{\dot
S^1(I)}^{\frac{d+3}{d+1}}+\|v\|_{\dot
S^1(I)}^{\frac{d+5}{d+1}}.\label{need}
\end{align}
This together with \eqref{air}, \eqref{bir} yields that
\begin{align*}
\|R(v)\|_{\dot N^1(I)}\lsm \|v\|_{\dot S^1(I)}^{\frac
{d+3}{d+1}}+\|v\|_{\dot S^1(I)}^{p_c}.
\end{align*}
Lemma \ref{super-rv} is proved.
\end{proof}

Based on this estimate, we have the following gain of decay
estimate.

\begin{lem}[Gain of decay]\label{gain-decay} Let $w^a=W^a-W$. Then
for sufficiently large $t$ we have
\begin{align}
&\|R(h+w^a)(t)-R(w^a)(t)\|_{L_x^{\frac{2d}{d+2}}}\lsm \|h(t)\|_{\dot
H_x^1}^{p_c}+\|h(t)\|_{\dot H_x^1}e^{-(p_c-1)e_0t},\label{gd-1}\\
&\|R(h+w^a)-R(w^a)\|_{\dot N^1([t,\infty))} \notag \\
&\qquad \lsm  \|h\|_{\dot
S^1([t,\infty))}^{p_c}+
\|h\|_{\dot S^1([t,\infty))}^{\frac {d+3}{d+1}}+
\|h\|_{\dot S^1([t,\infty))}e^{-(p_c-1)e_0
t}.\label{gn-dc-est}
\end{align}
\end{lem}
\begin{proof} Note from Corollary \ref{prop-wa} that
$|w^a(t,x)|\le\frac 14W(x)$, we have $w^a+W>\frac 12 W$. We can
write
\begin{align*}
&R(h+w^a)-R(w^a) \\
&=|h+w^a+W|^{p_c-1}(h+w^a+W)-|w^a+W|^{p_c-1}(w^a+W)-p_cW^{p_c-1}h\\
&=|h+w^a+W|^{p_c-1}(h+w^a+W)-|w^a+W|^{p_c-1}(w^a+W)-p_c|w^a+W|^{p_c-1}h\\
&\qquad+p_c
h(|w^a+W|^{p_c-1}-W^{p_c-1})\\
&=|w^a+W|^{p_c}J(\frac h{w^a+W})+p_c
h(|w^a+W|^{p_c-1}-W^{p_c-1})+|h|^{p_c-1}h,
\end{align*}
where $J(\cdot)$ is defined in \eqref{js-def}.
We first prove \eqref{gd-1}. Noting by \eqref{js} $|J(s)|\lsm |s|^{p_c}$ and using
Corollary \ref{prop-wa} we have
\begin{align*}
\|R(h+w^a)(t)-R(w^a)(t)\|_{L_x^{\frac{2d}{d+2}}}&\lsm
\||h|^{p_c}\|_{L_x^{\frac{2d}{d+2}}}+\|h|w^a|^{p_c-1}\|_{L_x^{\frac{2d}{d+2}}}\\
&\lsm \|h(t)\|_{\dot H_x^1}^{p_c}+\|h(t)\|_{\dot
H_x^1}\|w^a(t)\|_{L_x^{\frac{2d}{d-2}}}^{p_c-1}\\
&\lsm \|h(t)\|_{\dot H_x^1}^{p_c}+\|h(t)\|_{\dot
H_x^1}e^{-(p_c-1)e_0 t}.
\end{align*}
Next we prove \eqref{gn-dc-est}. By triangle inequality we have
\begin{align*}
&\|R(h+w^a)-R(w^a)\|_{\dot N^1([t,\infty))} \\
\lsm & \||\nabla|^{\frac
12}(|w^a+W|^{p_c}J(\frac{h}{w^a+W}))\|_{L_{t,x}^{\frac{2(d+1)}{d+3}}}\\
&+\|h(|w^a+W|^{p_c-1}-W^{p_c-1})\|_{L_t^1L_x^2}+\||\nabla|^{\frac
12}(|h|^{p_c-1} h)\|_{L_{t,x}^{\frac{2(d+1)}{d+3}}}.
\end{align*}
The third term on the RHS has been treated (see \eqref{bir}) so we have
\begin{align}
\||\nabla|^{\frac 12}(|h|^{p_c-1}
h)\|_{L_{t,x}^{\frac{2(d+1)}{d+3}}}\lsm \|h\|_{\dot
S^1([t,\infty))}^{p_c}.\label{col1}
\end{align}
For the second term, note for $d\ge 6$
\begin{align*}
||w^a+W|^{p_c-1}-W^{p_c-1}|\le |w^a|^{p_c-1},
\end{align*}
we estimate by using Corollary \ref{prop-wa}
\begin{align}
\|h(|w^a+W|^{p_c-1}-W^{p_c-1})\|_{L_s^1L_x^2}&\lsm
\|h\|_{L_s^2L_x^{\frac{2d}{d-3}}}\||w^a|^{p_c-1}\|_{L_s^2L_x^{\frac{2d}3}}\label{col2}\\
&\lsm e^{-(p_c-1)e_0 t}\|h\|_{\dot S^1([t,\infty))}.\notag
\end{align}
This is good for us.
To estimate the first term, we borrow the proof of \eqref{need}
which is still valid if we replace $W$ by another function having
the same decay. Indeed, by checking the proof of \eqref{need}, we
easily find that all we need is the following
\begin{align}
\||\nabla|^{\frac 12}(W^{\alpha}\psi_k)\|_p\lsm 2^{k(\frac
dp-\alpha(d-2)-\frac 12)},\ \ \forall\ \alpha>0,\label{hug1}\\
\|W^{\beta} \psi_k\|_p\lsm 2^{k(\frac dp-\beta(d-2))},\ \ \forall \
\beta\in\R,\label{hug2}
\end{align}
with the same estimates holding for $\tilde\psi_k$. Using Corollary
\ref{prop-wa}, we verify that \eqref{hug1}, \eqref{hug2} hold if we
replace $W$ by $w^a+W$. Thus, the same proof in proving \eqref{need}
establishes that
\begin{align}
\||\nabla|^{\frac 12}((w^a+W)^{p_c}J(\frac
v{w^a+W}))\|_{L_{t,x}^{\frac{2(d+1)}{d+3}}}\lsm \|h\|_{\dot
S^1([t,\infty))}^{\frac{d+3}{d+1}}+\|h\|_{\dot
S^1([t,\infty))}^{\frac{d+5}{d+1}}.\label{col3}
\end{align}
Collecting the estimates \eqref{col1}, \eqref{col2}, \eqref{col3},
we obtain \eqref{gn-dc-est}.

\end{proof}

\section{Classification of the solution}

Our purpose of this section is to prove Theorem \ref{u is w}.
Following the argument in \cite{duck-merle:wave}, the key step is to
establish the following

\begin{thm}\label{u-is-wa2}
Let $\gamma_0>0$. Assume $u$ is the solution of the equation in
\eqref{nlw} satisfying $E(u_0,u_1)=E(W,0)$ and
\begin{align}
\|\nabla(u(t)-W)\|_{2}+\|\partial_t u(t)\|_2\le Ce^{-\gamma_0 t},\ \
\forall \ t\ge 0 ,\label{small decay}
\end{align}
then there exists unique $a\in\R$ such that
$$
u=W^a.
$$
\end{thm}

As a corollary of Theorem \ref{u-is-wa2}, we see that modulo time
translation, all the $\{W^a, a>0\}$ (also $\{W^a,\ a<0\}$) are same.

\begin{cor}\label{cor-ab}
For any $a\neq 0$, there exists $T_a\in\R$ such that
\begin{equation}\label{define}
\begin{cases}
W^a(t)=W^+(t+T_a),\ \mbox{ if } a>0,\\
W^a(t)=W^-(t+T_a),\ \mbox{ if } a<0.
\end{cases}
\end{equation}
\end{cor}

We now prove Theorem \ref{u-is-wa2}. The strategy is the following:
we first prove that there exists $a\in \R$ such that
$\|\nabla(u(t)-W^a(t))\|_2+\|\partial_t(u(t)-W^a(t))\|_2$ has enough
decay, then using the decay estimate to show that $u(t)-W^a(t)$ is
actually identically zero. To this end, we have to input the
condition \eqref{small decay} and upgrade it to the desired decay
estimate. At this point, we need the following crucial result from
\cite{duck-merle:wave}.

\begin{lem}\label{upgrade}\footnote{The presentation of Lemma
\ref{upgrade} is slightly different from Proposition 5.7 in
\cite{duck-merle:wave}. Here we use a weaker condition \eqref{ud-1},
\eqref{ud-2} to yield stronger conclusions \eqref{arrive}, \eqref{not-arrive}.
However, one can easily find that this change is harmless once we
apply Strichartz estimate and repeat the same argument in
establishing Proposition 5.7 in \cite{duck-merle:wave}.} Let $t_0\ge
0$. Let $h$ be the solution to the equation
\begin{align}
\partial_{tt} h+\mathcal L h=\eps,\ t\ge t_0,\label{linear-eq}
\end{align}
where $h\in C^0([t_0,\infty);\dot H_x^1)$, $\partial_t h\in
C^0([t_0,\infty);L_x^2)$, $\eps\in \dot N^1([t_0,\infty))$. Assume
for some constant $c_0$, $c_1$ such that $0<c_0<c_1$,
\begin{align}
\|\nabla h(t)\|_2+\|\partial_t h(t)\|_2&\le Ce^{-c_0
t},\label{ud-1}\\
\|\eps(t)\|_{L_x^{\frac{2d}{d+2}}}+\|\eps\|_{\dot
N^1([t,\infty))}&\le Ce^{-c_1 t}.\label{ud-2}
\end{align}
Let $c_1^-$ be an arbitrary number smaller than $c_1$. Then the
following statements hold true,

$\bullet$ If $c_1> e_0$, there exists $A\in\R$ such that
\begin{align}
\|\partial_t(h(t)-Ae^{-e_0 t}\mathcal Y)\|_2+ \|\nabla(h(t)-Ae^{-e_0
t}\mathcal Y)\|_2+\|h-Ae^{-e_0 t}\|_{\dot S^1([t,\infty))}\le C
e^{-c_1^- t}.\label{arrive}
\end{align}

$\bullet$ If $c_1\le e_0$,
\begin{align}
\|\nabla h(t)\|_2+\|\partial_t h(t)\|_2+\|h\|_{\dot
S^1([t,\infty))}\le Ce^{-c_1^- t}.\label{not-arrive}
\end{align}
\end{lem}

\vspace{0.3cm}

We give the proof of Theorem \ref{u-is-wa2}.

\begin{proof}[Proof of Theorem \ref{u-is-wa2}]
We divide the proof into three steps.

\texttt{Step 1}. Let $v=u-W$, then condition \eqref{small decay} gives
that
\begin{align}
\|\nabla v(t)\|_2+\|\partial_t v(t)\|_2\le Ce^{-\gamma_0
t}.\label{fs0}
\end{align}
WOLOG, we can assume $\gamma_0 < e_0$. We first show that this decay rate can be upgraded to $e^{-e_0^{-} t}$.
More precisely, we will prove that
\begin{align}
\begin{cases}
\|\nabla v(t)\|_2+\|\partial_t v(t)\|_2\le C e^{-e_0^{-} t},\\
\|R(v)(t)\|_{L_x^{\frac{2d}{d+2}}}+\|R(v)\|_{\dot
N^1([t,\infty))}\le C e^{-\frac{d+3}{d+1}e_0^{-} t}.
\end{cases}\label{fs1}
\end{align}
And there exists $a\in \R$ such that $\forall \ \eta>0$
\begin{align}
\|\nabla(v(t)-ae^{-e_0 t}\mathcal Y)\|_2&+\|\partial_t(v(t)-a
e^{-e_0
t}\mathcal Y)\|_2\label{fs2}\\
&+\|v-ae^{-e_0 s}\mathcal Y\|_{\dot S^1([t,\infty))}\le C_{\eta}
e^{-(\frac{d+3}{d+1}-\eta)e_0 t}.\notag
\end{align}
Note \eqref{fs2} is a consequence of \eqref{fs1}. Indeed assume \eqref{fs1} is true, then since $v$
satisfies the equation
%\begin{align*}
%\partial_{tt}-\Delta v=p_c W^{p_c-1}v+R(v),
%\end{align*}
%or equivalently
\begin{align*}
\partial_{tt} v+\mathcal L v=R(v),
\end{align*}
applying Lemma \ref{upgrade} with $h=v$, $\eps=R(v)$ and $c_0=e_0^{-}$,
$c_1=\frac{d+3}{d+1}e_0^{-}$, we obtain \eqref{fs2}. It remains to show
\eqref{fs1} by using the condition \eqref{fs0}. To begin with we
show that \eqref{fs0} implies that
\begin{align*}
\|v\|_{\dot S^1([t,\infty))}\le C e^{-\gamma_0 t}.
\end{align*}
Let $s\ge t$. Let $\tau>0$ be a small constant to be chosen later. Using
Strichartz estimate on $[s,s+\tau]$ and Lemma \ref{super-rv} we have
\begin{align*}
\|v\|_{\dot S^1([s,s+\tau])}&\lsm \|\nabla v(s)\|_2+\|v_t(s)\|_2+
\|W^{p_c-1}v\|_{L_t^1L_x^2([s,s+\tau]\times\R^d)}+\|R(v)\|_{\dot
N^1([s,s+\tau])}\\
&\lsm e^{-\gamma_0 s}+\tau\|v\|_{L_t^{\infty}\dot
H_x^1([s,s+\tau]\times\R^d)}\|W^{p_c-1}\|_{L_x^d}+\|v\|_{\dot
S^1([s,s+\tau])}^{\frac{d+3}{d+1}}+\|v\|_{\dot S^1([s,s+\tau])}^{p_c}\\
&\lsm e^{-\gamma_0 s}+\tau\|v\|_{\dot S^1([s,s+\tau])}+\|v\|_{\dot
S^1([s,s+\tau])}^{p_c} + \| v\|^{\frac{d+3}{d+1}}_{\dot S^1([s,s+\tau])}.
\end{align*}
Taking $\tau$ small enough, by continuity argument we have
\begin{align*}
\|v\|_{\dot S^1([s,s+\tau])}\le C e^{-\gamma_0 s}.
\end{align*}
By triangle inequality we obtain
\begin{align*}
\|v\|_{\dot S^1([t,\infty))}&\le \sum_{j\ge 0} \|v\|_{\dot
S^1([t+\tau j,t+\tau(j+1)])}\\
&\le C\sum_{j\ge 0} e^{-\gamma_0(t+\tau j)}\\
&\lsm e^{-\gamma_0 t}.
\end{align*}
Lemma \ref{super-rv} yields that
\begin{align*}
\|R(v)(t)\|_{L_x^{\frac{2d}{d+2}}}\lsm e^{-\frac{d+3}{d+1}\gamma_0
t},\ \ \|R(v)\|_{\dot N^1([t,\infty))}\le
e^{-\frac{d+3}{d+1}\gamma_0 t}.
\end{align*}
Now we can apply Lemma \ref{upgrade} to obtain
\begin{align*}
\|\nabla v(t)\|_2+\|\partial_t v(t)\|_2+\|v\|_{\dot
S^1([t,\infty))}\le C(e^{-e_0^{-} t}+e^{-\frac{d+2}{d+1}\gamma_0 t}).
\end{align*}
If $\frac{d+2}{d+1}\gamma_0\ge e_0$, \eqref{fs1} is proved.
Otherwise, we are in the same situation with $\gamma_0$ being
replaced by $\frac{d+2}{d+1}\gamma_0$. Iterating this process
finitely many times yields \eqref{fs1}.

\vspace{0.2cm}

\texttt{Step 2}. In this step we prove that $u(t)-W^a(t)$ decays
arbitrarily fast. We prove that $\forall\ m>0$, there exists $t_m>0$
such that
\begin{align}
\|u-W^a\|_{\dot S^1([t,\infty))}\le e^{-m t}, \ t\ge
t_m.\label{dudu0}
\end{align}
To begin with, we show \eqref{dudu0} holds for $m=\frac{d+2}{d+1}e_0
$. Indeed, by triangle inequality and recall that $v=u-W$, we
estimate
\begin{align*}
&\|u-W^a\|_{\dot S^1([t,\infty))} \\
\le & \|v-ae^{-e_0 s}\mathcal
Y\|_{\dot S^1([t,\infty))}+\|w^a-v_{k_0}\|_{\dot
S^1([t,\infty))}+\|v_{k_0}-ae^{-e_0 s}\mathcal Y\|_{\dot
S^1([t,\infty))}.
\end{align*}
For the first term, we use \eqref{fs2} to estimate
\begin{align*}
\|v-ae^{-e_0 s}\mathcal Y\|_{\dot S^1([t,\infty))}\le \frac 13
e^{-\frac{d+2}{d+1}e_0 t}.
\end{align*}
For the last two terms, we use the definition of $v_{k_0}$ and
Corollary \ref{prop-wa} to get
\begin{align*}
\|w^a-v_{k_0}\|_{\dot S^1([t,\infty))}&\le e^{-k_0 e_0 t}\le \frac
13 e^{-\frac{d+2}{d+1}e_0 t}.\\
\|v_{k_0}-ae^{-e_0 s}\mathcal Y\|_{\dot
S^1([t,\infty))}&\le\|\sum_{2\le j\le k_0}e^{-je_0 t}\Phi_j\|_{\dot
S^1([t,\infty))}\\
&\lsm e^{-2e_0 t}\le \frac 13e^{-\frac{d+2}{d+1}e_0 t}.
\end{align*}
Collecting these estimates together we obtain
\begin{align*}
\|u-W^a\|_{\dot S^1([t,\infty))}\le e^{-\frac{d+2}{d+1}e_0 t}.
\end{align*}

Now suppose \eqref{dudu0} hold for some $m_1\ge \frac{d+2}{d+1}e_0$,
we show \eqref{dudu0} holds for $m=m_1+\frac 1 {d+1} e_0$. This
will yield \eqref{dudu0} by iteration.

Write $h=u-W^a$, then $h$ satisfies
\begin{align*}
\partial_{tt}h+\mathcal L h=R(h+w^a)-R(w^a).
\end{align*}
Since by Lemma \ref{gain-decay}, we have
\begin{align*}
\|R(h+w^a)(t)-R(w^a)(t)\|_{L_x^{\frac{2d}{d+2}}}&\lsm \|h(t)\|_{\dot
H_x^1}^{p_c}+\|h(t)\|_{\dot H_x^1} e^{-(p_c-1)e_0 t}\\
&\lsm e^{-(m_1+(p_c-1)e_0)t}\\
 \|R(h+w^a)-R(w^a)\|_{\dot
N^1([t,\infty))}& \lsm \|h\|_{\dot
S^1([t,\infty))}^{p_c}+e^{-(p_c-1)e_0
t}\|h\|_{\dot S^1([t,\infty))} + \| h\|^{\frac{d+3}{d+1}}_{\dot S^1([t,\infty))} \\
&\lsm e^{-(m_1+\frac 2{d+1} e_0 )t},
\end{align*}
Applying Lemma \ref{upgrade} again gives us
\begin{align*}
\|h\|_{\dot S^1([t,\infty))}\le e^{-(m_1+\frac 1{d+1}e_0)t}.
\end{align*}

\texttt{Step 3}. Now we show there exists $m>0$ such that $h(t)=0$
for all $t\ge t_m$ by using decay estimate \eqref{dudu0}.

First we note $h$ satisfies the equation
\begin{align*}
\partial_{tt}h-\Delta h =p_c W^{p_c-1} h+R(w^a+h)-R(w^a),
\end{align*}
and the following Duhamel's formula:
\begin{align*}
h(t)=-\int_t^{\infty}\frac{\sin((t-s)\sqrt{-\Delta})}{\sqrt{-\Delta}}
(p_cW^{p_c-1}h+R(W^a+h)-R(W^a))(s)ds .
\end{align*}
Applying Strichartz estimate on $[t,\infty)$, we obtain
\begin{align*}
\|h\|_{\dot S^1([t,\infty))}\le C(
\|W^{p_c-1}h\|_{L_s^1L_x^2([t,\infty)\times\R^d)}+\|R(h+w^a)-R(w^a)\|_{\dot
N^1([t,\infty))}).
\end{align*}
Denote $\|h\|_{\Sigma_t}:=\sup_{s\ge t}e^{ms}\|h\|_{\dot
S^1([s,\infty))}$. Then for $\eta>0$ small enough we have
\begin{align*}
\|W^{p_c-1}h\|_{L_s^1L_x^2([t,\infty)\times\R^d)}&\lsm \sum_{j\ge
0}\|W^{p_c-1}h\|_{L_s^1L_x^2([t+\eta j,t+\eta(j+1)]\times\R^d)}\\
&\le \eta\|W^{p_c-1}\|_{L_x^{
d}}\|h\|_{L_s^{\infty}L_x^{\frac{2d}{d-2}}([t+\eta
j,t+\eta(j+1)]\times\R^d)}\\
&\lsm \sum_{j\ge 0} \eta\|h\|_{\dot S^1([t+\eta j,t+(j+1)\eta])}\\
&\lsm \sum_{j\ge 0} \eta e^{-m(t+\eta j)}\|h\|_{\Sigma_{t_m}}\\
&\lsm e^{-mt}\|h\|_{\Sigma_{t_m}}\frac{\eta}{1-e^{-\eta m}}\\
&\lsm \frac 2 m e^{-mt}\|h\|_{\Sigma_{t_m}}.
\end{align*}
From Lemma \ref{gain-decay}, we get
\begin{align*}
\|R(w^a+h)-R(w^a)\|_{\dot N^1([t,\infty))}\le \frac
1{10C}e^{-mt}\|h\|_{\Sigma_{t_m}}.
\end{align*}
Combining these two estimates, we get
$$
\|h\|_{\Sigma_{t_m}}\le \frac 12\|h\|_{\Sigma_{t_m}},
$$
which implies that $h(t)=0$ on $[t_m,\infty)$. Recall that
$h(t)=u(t)-W^a(t)$ we obtain $u(t)=W^a(t)$ on  $[t_m,\infty)$.
Therefore $u\equiv W^a$ by uniqueness of solutions to \eqref{nlw}.
The Proposition is proved and we have Theorem \ref{u-is-wa2}.
\end{proof}

\begin{proof} [Proof of Corollary \ref{cor-ab}]
The proof is almost the same as
Corollary 6.5 in \cite{duck-merle:wave}. Let $a\neq0$ and $T_a$ be
such that $|a|e^{-e_0 T_a}=1$. By \eqref{extra-fir} we have
\begin{equation}\label{change}
\|W^a(t+T_a)-W\mp e^{-e_0t}\mathcal Y\|_{H^{m,m}}\lesssim
e^{-\frac 32 e_0t}.
\end{equation}
Moreover $W^a(\cdot+T_a)$ satisfies the assumption in Theorem
\ref{u-is-wa2}, thus there exists $a'$ such that $W^a(\cdot
+T_a)=W^{a'}$. By \eqref{change}, $a'=1$ if $a>0$ and $a'=-1$ if
$a<0$. Corollary \ref{cor-ab} is proved.

\end{proof}

Finally, we give the proof of the main theorem \ref{u is w}.

\vspace{0.3cm}

\emph{Proof of Theorem \ref{u is w}:} We first note that (b) is just
the variational characterization of $W$. More precisely we have
\begin{thm}\label{w-like}\cite{aubin,talenti}
Let $c(d)$ denote the sharp constant in Sobolev embedding
$$
\|f\|_{\frac{2d}{d-2}}\le c(d)\|\nabla f\|_2.
$$
Then the equality holds iff $f$ is $W$ up to symmetries. More
precisely, there exists $\lambda_0>0$, $x_0 \in \R^d$, $\delta_0 \in \{-1,+1\}$,
such that
$$
f(x)=\delta_0 \lambda_0^{-\frac{d-2}2}W(\frac
{x-x_0}{\lambda_0}).
$$
In particular, if
 $(u_0, u_1)\in \dot H_x^1\times L_x^2$ satisfies
$$
E(u_0, u_1)=E(W,0),\ \|\nabla u_0\|_2=\|\nabla W\|_2,
$$
then $(u_0,u_1)=(W,0)$ up to symmetries, hence the corresponding
solution $u$ coincides with $W$ up to symmetries.
\end{thm}

It remains for us to show (a), (c). We first prove (a). Let $u$ be
the maximal-lifespan solution of \eqref{nlw} on $I$ satisfying
$E(u_0, u_1)=E(W,0)$, $\|\nabla u_0\|_2<\|\nabla W\|_2$. Then by
Proposition \ref{prop:exp}, we have $I=\R$.  Assume that $u$ blows
up forward in time. Applying Proposition \ref{prop:exp} again, we
conclude that there exist $x_0 \in \R^d$, $\mu_0$, $\gamma_0$, $C>0$ such that
\begin{align*}
\|\nabla(u(t)-W_{[\mu_0,x_0]})\|_{2}+\|\partial_t u(t)\|_2\le
Ce^{-\gamma_0 t}.
\end{align*}
where $W_{[\mu_0,x_0]}=\mu_0^{-\frac{d-2}2}W(\frac{x+x_0}{\mu_0})$.
This implies
\begin{align*}
\|\nabla(u_{[\mu_0^{-1},-\mu_0^{-1} x_0]}(t)-W)\|_{2}+\|\partial_t(u_{[\mu_0^{-1},- \mu_0^{-1} x_0]}(t))\|_2\le
e^{-\gamma_0\mu_0 t}
\end{align*}
where
$$
u_{[\mu_0^{-1},-\mu_0^{-1} x_0]}(t,x)=\mu_0^{\frac{d-2}2}u(\mu_0 t,\mu_0
x-x_0)
$$
is also a solution of the equation \eqref{nlw}. By Theorem
\ref{u-is-wa2} with $\gamma_0$ now replaced by $\gamma_0\mu_0$, we
conclude there exists $a<0$ such that $u_{[\mu_0^{-1},- \mu_0^{-1} x_0]}=W^a$.

Using Corollary \ref{cor-ab}, we get
$$
u(t,x)=\mu_0^{-\frac{d-2}2}W^-(\mu_0^{-1}t+T_a,\mu_0^{-1}(x+x_0)).
$$
This shows that $u=W^-$ up to symmetries. The proof of (c) is
similar so we omit it. This ends the proof of Theorem \ref{u is w}.

\end{document}